\documentclass[a4paper, 12pt]{article}
\usepackage[amsmath]{e-jc}

\usepackage[british]{babel}
\usepackage[utf8]{inputenc}
\usepackage[T1]{fontenc}

\usepackage{graphicx}
\usepackage[dvipsnames]{xcolor}
\usepackage{hyperref}
\usepackage{algorithm}

\usepackage[noend]{algpseudocode}
\usepackage{array}

\newcommand \esperluette {&}

\newlength{\textlarg}

\title{Cayley trees and increasing 1,2-trees: let's twist!}

\dateline{Mar 18, 2025}{TBD}{TBD}
\Copyright{The authors. Released under the CC BY-ND license (International 4.0).}

\DeclareMathOperator*{\DisjointUnion}{\text{\scalebox{1.3}{$\uplus$}}}

\DeclareMathOperator*{\Cayley}{Cayley}

\newcommand{\pdiff}[2]{ \frac{\partial #1}{\partial #2} }

\newcommand{\fig}[3]{
\begin{figure}[htb]
\begin{center}
 \includegraphics #1
 \end{center}
\vspace{-7pt}
\caption{#2}
\label{#3}
\end{figure}
}

\author{Julien Courtiel\authornote{1}, Matthieu Dien\authornote{1}, Paul Dorbec\authornote{1}}

\authortext{1}{Université Caen Normandie, ENSICAEN, CNRS, Normandie Univ, GREYC UMR 6072, F-14000 Caen, France}

\newcommand{\vertex}[1]{\mathbf{#1}}

\newcommand{\lektree}{$(\le k)$-tree}
\newcommand{\lektrees}{\text{$(\le k)$-trees}}

\newcommand{\eddy}[1]{ \textsc{Eddy}(#1) }

\newcommand{\taut}{$\boldsymbol \tau$}

\begin{document}

\maketitle

\begin{abstract}
An increasing 1,2-tree is a labeled graph formed by starting with a vertex and then repeatedly attaching a leaf to a vertex or a triangle to an edge, the labeling of the vertices corresponding to the order in which the vertices are added.
Equivalently, increasing 1,2-trees are connected chordal graphs of treewidth at most 2 labeled with a reversed perfect elimination ordering.

We prove that this family is equinumerous with Cayley trees, which are unconstrained labeled trees.
In particular, the number of triangles in an increasing 1,2-tree corresponds to the number of twists. A twist (also called improper edge) is an edge whose endpoint closer to vertex $\mathbf 1$  has a greater label than some vertex in the subtree rooted at the other endpoint of the edge.

We provide three proofs of this result, the first being based on similar recursive decompositions, the second on the resolution of generating functions, and the third describing a bijection. Finally, we propose efficient random generators for these two combinatorial families.

\end{abstract}

\textbf{keywords:} chordal graphs, bounded treewidth, Cayley trees, generating function, bijection, random generation.

\section{Introduction}

\subsection{Motivation}

Chordal graphs, also known as triangulated graphs, are graphs where every cycle of length at least four admits a chord,
i.e. an edge joining two non-consecutive vertices on the cycle.
They form a well-studied class of graphs appearing in a variety of domains (for example in numerical linear algebra \cite{linear_system}, in Bayesian networks \cite{bayesian, bayesian2}, in phylogenetics \cite{phylogeny}), each benefiting from their unique properties and efficient algorithms for processing.
Both these appearance domains and these efficient algorithms generally rely on the fact that chordal graphs admit a special vertex ordering called a \emph{perfect elimination ordering}.

Because of their ubiquity, it seems crucial to design random generators for chordal graphs. 
This is borne out by the abundant literature on the subject 
since Wormald's seminal work~\cite{wormald}.  
Numerous algorithmic approaches have been developed for generating random chordal graphs \cite{algo_bidon,markov-chain,turcs2,trois-algos,turcs}, 
but these algorithms lack formal guarantees regarding the distribution of their output. 
More recently, a polynomial-time uniform sampler for chordal graphs (tractable up to 30 vertices) was proposed in \cite{polynomial-sampling}.

The uniform distribution is not necessarily the most interesting one since chordal graphs of large size almost surely have the same shape \cite{split}: about half the vertices form a clique, the rest form an independent set.
This is why it seems worthwhile to restrict random generation to a subclass of chordal graphs. 
It makes sense to examine  chordal graphs with bounded treewidth,
as this plays an important role in many applicative contexts.

A well-studied subclass of chordal graphs with bounded treewidth
is the family of $k$-trees~\cite{origin-k-trees,number-k-trees,partial-k-tree},
which are often described in terms of perfect elimination ordering.
A perfect elimination ordering is an ordering of the vertices such that
for each vertex $v$, its neighbors occuring after $v$ in the order form a clique.
If those cliques are all of size $k$ (except for the last $k$ vertices),
then the chordal graph is called a $k$-tree.
Alternatively, $k$-trees are also defined as maximal graphs with treewidth $k$
(i.e. adding any edge to a $k$-tree increases its treewidth).
Equivalently, the graphs with treewidth at most $k$
are all the subgraphs of $k$-trees (called \emph{partial $k$-trees}).

In the following, we consider \emph{\lektrees},
which are graphs admitting a perfect elimination ordering
where the neighborhoods are cliques of size at most $k$
(instead of exactly $k$ for $k$-trees).
Equivalently, the \lektrees\ are the connected chordal partial $k$-trees,
or our natural family of connected chordal graphs with treewidth at most $k$.
It should be noted that the sets of $k$-trees do not cover all chordal graphs,
whereas \lektrees\ do.

Exact and asymptotic counting of \lektrees\
have been conducted in ~\cite{enumeration-chordal,limit,partial-k-tree}.
In our work, we consider labeled \lektrees\ where the labeling
induces a perfect elimination ordering.
This seems more meaningful in many contexts.
This leads to the notion of \emph{increasing} \lektrees, defined in the next subsection.

The present paper only focuses on the case $k=2$, for which we have discovered a surprising result: 
the number of increasing $(\le 2)$-trees of size $n$ is given by $n^{n-2}$, which is also the number of Cayley trees on $n$ vertices. 
Our article aims to explain why, in the expectation that some of the results will transfer to increasing \lektrees, for $k \geq 3$.

\subsection{Definitions and results}

In this paper, we consider simple undirected graphs whose vertices are labeled from $1$ to $n$, where $n$ is the number of vertices. 
We denote $\vertex x$ the vertex with label $x$.

Recall that chordal graphs, besides being graphs with no chordless cycle, are graphs admitting a perfect elimination ordering,
that is an ordering of the vertices $v_1,\ldots,v_n$ such that the neighborhood of $v_{i}$ in the  subgraph of $G$ induced by the vertex set $v_{i+1},\ldots, v_n$ is a complete subgraph.

\begin{definition}[Increasing \lektrees] \label{def:lektree}
Given a positive integer $k$, increasing \lektree s are labeled graphs inductively defined as follows:
\begin{itemize}
 \item the only 1-vertex increasing \lektree\ is an isolated vertex $\vertex 1$.
 \item given a $n$ vertices increasing \lektree\ $T$, consider a clique of size at most $k$ and attach vertex  $\vertex{n+1}$ adjacent to all vertices of the clique.
\end{itemize}
\end{definition}

Note that an increasing \lektree\ is a \lektree\ whose labeling
is a reverse perfect elimination ordering.

In the following, we mostly study $(\le 2)$-trees, that we name 1,2-trees. There are 16 increasing 1,2-trees of size $4$, as shown by Figure~\ref{16increasing}.
Observe that the top rows contain 1-trees (or just trees),
the bottom row contains 2-trees, and in between,
there are $(\le 2)$-trees which are not $k$-trees.

\fig{[width=0.75 \textwidth]{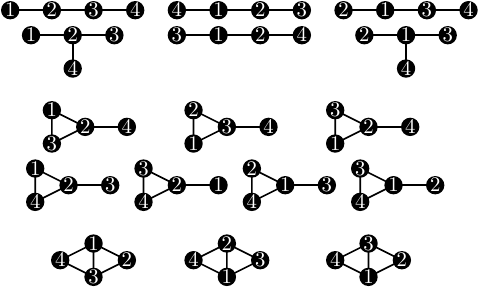}}{The 16 increasing 1,2-trees with $4$ vertices, gathered by their numbers of edges.}{16increasing}

The second family of combinatorial objects we consider consists of Cayley trees:

\begin{definition}[Cayley trees]
 Cayley trees are trees with a labeling of the vertices from $\vertex 1$ to $\vertex n$, where $n$ is the number of vertices in the tree. Those trees are naturally rooted in $\vertex 1$.
\end{definition}

Cayley trees are just trees.
The noun adjunct \textit{Cayley} is there to insist that the vertex labeling has no constraint.
This name originates from the mathematician Arthur Cayley who found that the number of trees with $n$ labeled vertices is $n^{n-2}$ \cite{cayley}.
Figure~\ref{16cayley} lists the 16 Cayley trees with $4$ vertices.

\fig{[width=0.7 \textwidth]{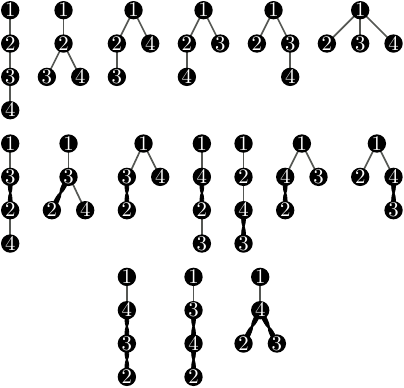}}{The 16 Cayley trees with $4$ vertices. The trees are sorted by the number of twists
($0$ twist in the first row, $1$ twist in the second row, $2$ twists in the bottom row). 
The twists are represented differently (in bold and as if the edges have actually been twisted)  than the increasing edges.  }{16cayley}

By rooting the Cayley trees in $\vertex 1$, we induce a notion of lineage (\emph{child}, \emph{parent}).
A \emph{descendant} of a vertex is either a child of the vertex or a child of some descendant of the vertex.
When we talk about \emph{subtrees}, we imply that they comprise a vertex and all its descendants.

\begin{definition}[increasing edge, twist]\label{def:twist}
Let $T$ be a Cayley tree rooted at $\vertex 1$,
 and $e$ be an edge of $T$ from a vertex $\vertex x$ to a subtree $S$.
 We say that $e$ is \emph{increasing} if all vertices in $S$ have a label greater than $x$.
 Otherwise, $e$ is called a \emph{twist}.
\end{definition}

Consider for example the linear tree $\vertex 1 - \vertex 3 - \vertex  6 - \vertex 2 - \vertex 5 - \vertex 4$.
 The edge between $\vertex 2$ and $\vertex 5$ is increasing: it connects $\vertex 2$ to the subtree $\vertex 5 - \vertex 4$.
 On the contrary, the edge between $\vertex 3$ and $\vertex 6$ is a twist: the subtree rooted at $\vertex 6$ contains $\vertex 2$ whose label is smaller than $3$.

Twists were called \emph{improper edges} in \cite{sho95}.

\begin{definition}[$\min(S)$]
Let $S$ be a subtree in a Cayley Tree. The \emph{minimum} of $S$, denoted by $\min(S)$, is the minimal label contained in $S$.
\label{def:min}
\end{definition}

Note that an edge is increasing if and only if it links a vertex $\vertex x$ to a subtree $S$ such that $x < \min(S)$.
Conversely, it is a twist if and only if $\min(S) < x$.

In this paper, we prove the following 
result:

\begin{theorem}
Increasing 1,2-trees with $n$ vertices and $m$ edges are in bijection with Cayley trees with $n$ vertices and $m-n+1$ twists.
\label{theo:central}
\end{theorem}

The number of increasing 1,2-trees can thus be straightforwardly deduced from Cayley's formula \cite{cayley}.

\begin{corollary}
There are $n^{n-2}$ increasing 1,2-trees with $n$ vertices.
\label{cor:nn-2}
\end{corollary}

Note also that the number of triangles in a 1,2-tree is equal to $m-n+1$
since a $n$-vertices tree has $n - 1$ edges and in a 1,2-tree,
adding a vertex adjacent to an edge (a triangle) adds two edges instead of one.
This could also be inferred from Euler's formula,
as 1,2-trees are planar and their number of triangles
is equal to the number of inner faces,
though all faces are not triangles.

This paper provides three different proofs of Theorem~\ref{theo:central}.
In Section~\ref{s:recursive}, we show that
the recursive decomposition following
from the definition of increasing 1,2-trees is isomorphic
to Shor's decomposition~\cite{sho95} of Cayley trees.
In Section~\ref{s:gf}, we find the generating function of increasing 1,2-trees
by solving a Partial Differential Equation
and show that it can be expressed in terms of the generating function of Cayley trees.
In Section~\ref{s:bij}, we describe an explicit bijection, using as an intermediary object a labeled forest where each edge is a twist.
Finally, as an application, we present in Section~\ref{s:random_generation}
a linear uniform random generator for increasing 1,2-trees of size $n$
running in $O(n)$ time.

\section{Proof by recursive decomposition}
\label{s:recursive}

This section shows that increasing 1,2-trees and Cayley trees have similar recursive decompositions.
This gives a first inductive proof of Theorem~\ref{theo:central}.

\subsection{Recurrence for the number of increasing 1,2-trees}

The recurrence for increasing 1,2-trees follows from Definition~\ref{def:lektree}:

\begin{proposition}Let $c_{n,m}$ be the number of increasing 1,2-trees with $n$ vertices and $m$ edges. The numbers satisfy the recurrence
\[ c_{n,m} = (n-1) c_{n-1,m-1} + (m-2) c_{n-1,m-2},  \]
with initial condition $c_{1,0} = 1$ and $c_{1,k} = 0$ for $k \neq 0$.
\label{prop:rec}
\end{proposition}

\begin{proof}
There are two ways to construct an increasing 1,2-tree with $n$ vertices from an increasing 1,2-tree $T$ with $n-1$ vertices:
\begin{itemize}
\item choose any vertex from $T$ and attach $\vertex n$ to it,
\item choose any edge $e$ from $T$ and form a triangle between $\vertex n$ and the endpoints of $e$.
\end{itemize}
We cannot obtain two identical 1,2-trees with this process (since the neighbors of $\vertex n$ are always different),
 and it is possible to recover $T$ by removing $\vertex n$ and all its adjacent edges.
 The recurrence follows.
\end{proof}

\subsection{Recurrence for the number of Cayley trees}

The recursive decomposition for Cayley trees in terms of twists is already known.
Originally, Meir noticed in \cite{shor-meir} that for every number $z$, the recursion relation
\[S_z(x+1,y) = (x+z)S_z(x,y) + (x+y)S_z(x,y-1)\]
has a solution satisfying
\[\sum_{y=1}^x S_z(x,y) = (x+z)^x.\]
(This is a generalization of a formula found by Ramanujan \cite{ramanujan} -- even though Meir got inspired not by the Indian mathematician, but from a recreational mathematical problem created by Shor.)
With the notational change $c_{n,m} = S_{2}(n-2,m-n+1)$, we recover the recurrence for increasing 1,2-trees (see Proposition~\ref{prop:rec}). So from Meir's observation, we deduce that the number of increasing 1,2-trees satisfies
\[\sum_{m=n-1}^{2n-3} c_{n,m} = n^{n-2},\]
which is consistent with Corollary~\ref{cor:nn-2}.

A combinatorial interpretation of this result in terms of Cayley trees has been independently found by Shor \cite{sho95} on the one hand, and Dumont and Ramamonjisoa \cite{malgaches} on the other hand.
It translates as follows.

\begin{proposition}Let $c_{n,m}$ be the number of Cayley trees with $n$ vertices and $m-n+1$ twists. The numbers satisfy the recurrence
\[ c_{n,m} = (n-1) c_{n-1,m-1} + (m-2) c_{n-1,m-2},  \]
with initial condition $c_{1,0} = 1$ and $c_{1,k} = 0$ for $k \neq 0$.
\label{prop:rec_cayley}
\end{proposition}

\begin{figure}[htb]\begin{center}

\begin{tabular}{|c|c|}  \firsthline

 \begin{minipage}{0.15 \textwidth}
\begin{center}
The number of twists remains the same.
\end{center}
\end{minipage}

 &
 \begin{minipage}{0.75 \textwidth}
 \vspace{3pt}
\begin{center}
Choose a vertex $\vertex k$ and attach a leaf $\vertex n$ : \\
 \includegraphics[scale=1]{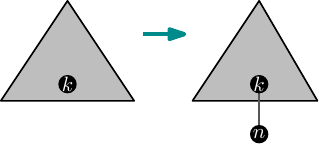}
\end{center}
 \end{minipage}
  \\
\hline
 \begin{minipage}{0.15 \textwidth}
\begin{center}
The number of twists increases.
\end{center}
\end{minipage}

&

 \begin{minipage}{0.75 \textwidth}
 \vspace{3pt}
\begin{center}
Choose a non-root vertex $\vertex k$. Let $T_1, \dots, T_b, T'_1, \dots, T'_c$ be the subtrees of $\vertex k$ arranged in such a way that
\[ \min(T_1) < \dots < \min(T_b) < k < \min(T'_1) < \dots < \min(T'_c). \]
Choose $a$ with $0 \leq a \leq b$, put $\vertex n$ in place of $\vertex k$, set $\vertex k$ as a child of $\vertex n$ and distribute the subtrees as follows: \\ \vspace{3pt}
 \includegraphics[scale=1]{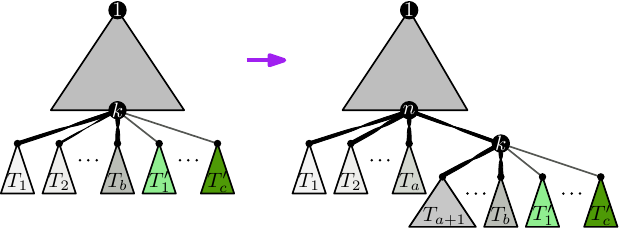} \vspace{3pt}
\end{center}

 \end{minipage}

\\ \hline
\end{tabular}
\end{center}

\caption{Shor's idea to enumerate all Cayley trees while controlling the number of twists \label{fig:shor-idea}}
\end{figure}

\begin{proof}
Even though it is a well-known result, the authors make the choice of rewriting here Shor's proof (for the sake of self-completeness). The idea is summarized by Figure~\ref{fig:shor-idea}.

Let $T$ be a Cayley tree with $n-1$ vertices. We build a new Cayley Tree $T'$ with $n$ vertices in one of the following ways:
\begin{itemize}
\item choose any vertex $\vertex k$ and attach $\vertex n$ to it as a leaf. The new edge is obviously increasing.
\item choose any vertex $\vertex k$ different from $\vertex 1$. We denote by $T_1,\dots,T_b,T'_1,\dots,T'_c$ the subtrees rooted at a child of $\vertex k$, sorted such that
\[ \min(T_1) < \dots < \min(T_b) < k < \min(T'_1) < \dots < \min(T'_c). \]
Thus, $b$ is the number of twists leaving $\vertex k$ and $c$ is the number of increasing edges leaving $\vertex k$.

First, change $\vertex k$ by $\vertex n$ and add $\vertex k$ as a child of $\vertex n$.

 Then, choose any $a \in \{0,\dots,b\}$. Leave $T_1, \dots, T_a$ attached to $\vertex n$ and move the subtrees $T_{a+1}, \dots,T_b,T'_1,\dots,T'_c$ as children of $\vertex k$.

 By doing so, we create a new twist between  $\vertex n$ and $\vertex k$
 and since we replace $\vertex k$ by a vertex with a larger label,
 all twists formerly attached to $\vertex k$ remain twists.
\end{itemize}
\begin{remark}\label{rk:method2}
After the addition of $\vertex n$ with the second method, the subtree rooted in $\vertex k$ is the subtree $S$
such that $\min(S)$ is maximal among all subtrees rooted at a child of $\vertex n$.
Indeed, $\min(S)$ is larger than $\min(T_i)$ for $i \le a$ since $\min(S) = \min(T_{a+1})$ if $a < b$ and $\min(S) = k$ if $a=b$.
 \end{remark}

We need now to make sure that we generate every Cayley tree of size $n$ exactly once. Consider then a Cayley tree $T$ with $n$ vertices.

If $\vertex n$ is a leaf, then there is a unique way to create $T$, that is applying the first method in $T-\vertex n$ on the parent of $\vertex n$.

If $\vertex n$ is not a leaf, then it can only be obtained with the second method.
Every edge leaving $\vertex n$ is a twist. Let $a+1$ be the number of such edges, and $S$ be the subtree attached to $\vertex n$ whose quantity $\min(S)$ is maximal over all these subtrees. We call $\vertex k$ the root of $S$.
Note that by Remark~\ref{rk:method2}, $\vertex k$ is the only possible choice for applying the second method to get $T$.

To recover the Cayley tree with $n-1$ edges, we contract the edge between $\vertex n$ and $\vertex k$ (we keep $\vertex k$ for the label of the merged vertex).
This corresponds to the reverse direction of the second method.
 Note that all of the $a$ twists that were leaving $\vertex n$ (excluded the edge we contracted) are still twists in the smaller tree,
 and the minima of the associated subtrees are smaller than the minima of the other subtrees attached to $\vertex k$.
  This is consistent with our choice in the second construction to attach the subtrees $T_1,\dots,T_a$ to vertex $\vertex n$ for some $a \in \{0,\dots,b\}$, and the other subtrees to $\vertex k$.

Finally, let us count the number of different Cayley trees $T'$ with $n$ vertices and $m-n+1$ twists we generate:
\begin{itemize}
\item The number of trees generated from the first method is $(n-1) \times c_{n-1,m-1}$, since $T$ must have $n-1$ vertices and $m-n+1 = (m-1) - (n-1) + 1$ twists.
\item  The number of ways of choosing a non-root vertex $\vertex k$ and an integer $a$ between $0$ and $b_k$, which is the number of twists adjacent to $\vertex k$ in $T$, is
\begin{equation} \label{eq:nbtwist}
 \sum_{\vertex k \neq \vertex 1} (1 + b_k) = (n-2) +  \sum_{\vertex k \neq \vertex 1} b_k\,.
 \end{equation}
Since no twist leaves $\vertex 1$,
$\sum_{\vertex k \neq \vertex 1} b_k$ is the total number of twists in $T$
which is $(m-n+1)-1 = m-n$. So the right operand in \eqref{eq:nbtwist} is equal to $m-2$. In total, the number of trees generated from the second approach is $(m-2) \times c_{n-1,m-2}$.
\end{itemize}
Summing the numbers for the two methods, we obtain the wanted recurrence.
\end{proof}

\subsection{Bijection(s) via generating trees}

Once a recursive decomposition of a combinatorial family such as above is found, it is easy to generate  exhaustively all objects up to a certain size.
 Formally, we can build a \emph{generating tree}, that is an infinite tree where the root is the only object of size $1$ and the children of a node are the objects that can be constructed from the node thanks to the recursive decomposition.

Figure~\ref{generating_trees} shows the top of the  generating trees of  increasing 1,2-trees and the genera\-ting tree of Cayley trees.

\fig{[width=0.95\textwidth]{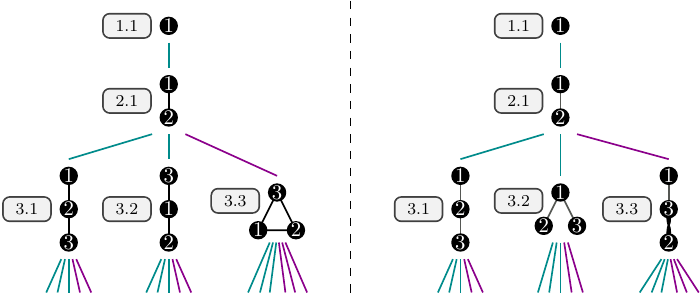}}{The first three levels of the generating tree of increasing 1,2-trees \textit{(left)} and the generating tree of Cayley trees \textit{(right)}.}{generating_trees}

If the recursive decompositions of two combinatorial families are isomorphic (here like increasing 1,2-trees and Cayley trees -- cf Propositions~\ref{prop:rec} and~\ref{prop:rec_cayley}), then the generating trees are identical.
 A natural bijection between the two families, albeit an artificial one, can be obtained by numbering the nodes of the two generating trees in the same way (for example with a breadth-first traversal). Figure~\ref{generating_trees} illustrates how it can be done with increasing 1,2-trees and Cayley trees.

Such a bijection may not be very satisfactory, in particular since the ordering of the children in a generating tree is not quite explicit.

In Section~\ref{s:bij} we are going to describe a more explicit bijection, which is
 unambiguous and has nice global properties.


\section{Proof by generating functions}
\label{s:gf}

In this section, we find a formula for the generating function of 1,2-trees  using a more analytic approach. 
Remark that this method should be generalizable to increasing \lektrees\  for every integer \(k \geq 1\).

More precisely, we prove the following:

\begin{theorem}
The exponential generating function of increasing 1,2-trees 
\[C(z,u) := \sum_{n,m \geq 0} c_{n,m} \frac{z^n u^m}{n!}\] counted by vertices (variable $z$) and edges (variable $u$) is given by
\begin{equation}
  \label{eq:Czu}
  C(z,u) = \Cayley\left( \frac{u z + u - 1}{u} e^{\frac{1-u}{u}} \right) 
   + \frac 1 {2 u^2} - \frac 1 2
  \end{equation}
  where $\Cayley(z)$ is the exponential generating function of (unrooted) Cayley trees counted by the vertices, that is \(\Cayley(z) = \sum_{n \geq 0} {n ^{n -2}} \frac{z^n} {n!} \).
\label{theo:gf}
\end{theorem}

By plugging $u=1$ in the previous equation, we obtain that $C(z,1) = \Cayley(z)$ which gives another proof that 1,2-trees and Cayley trees are counted by the same numbers.

\begin{remark}
  If we add a variable $t$ in $\Cayley$ to count the number of twists, 
  since \( \Cayley(z,t) = t C(\frac z t,t) \) by Theorem~\ref{theo:central}, 
  we can express the bivariate generating function $\Cayley(z,t)$ in terms of the univariate generating function $\Cayley(z)$ thanks to the previous theorem:
  \[
    \Cayley(z,t) = t \Cayley\left( \frac{ z + t - 1}{t} e^{\frac{1 - t}{t}  }  \right) + \frac 1 {2t} - \frac t 2. 
  \]  
  It should be interesting to find a combinatorial interpretation to this formula. We leave this as an open question.
\label{rem:interpretation}
\end{remark}

In order to prove Theorem~\ref{theo:central}, let us start by giving an equation satisfied by $C(z,u)$.

\begin{lemma}
The exponential generating function of increasing 1,2-trees $C(z,u)$ is solution of the Partial Differential Equation (PDE)
\begin{equation}
  \frac{\partial C}{\partial z} = 1 + z u \frac{\partial C}{\partial z} + u^3 \frac {\partial C}{\partial u} 
  \label{eq:pde}
\end{equation}
with the initial condition \(C(0,u) = 0\).
\label{lem:pde}
\end{lemma}

\begin{proof}
This PDE can be derived quite straightforwardly from the recurrence from Proposition~\ref{prop:rec}.

We can also obtain this PDE with the help of symbolic method (see~\cite{flajolet-sedgewick}). 
Indeed, the class $\mathcal C$ of increasing 1,2-trees admits the combinatorial specification
\[
  \mathcal{C} = \quad \mathcal Z \quad  \DisjointUnion \quad \mathcal Z \mathcal U \times \Theta_{\mathcal Z} \mathcal C \quad \DisjointUnion \quad \mathcal Z \mathcal U^2 \times \Theta_{\mathcal U} \mathcal C
  \]
where $\mathcal Z$ and $\mathcal U$ are the atoms for the vertices and the edges, 
$\DisjointUnion$ is the disjoint union, 
and $\Theta_{\mathcal Z}$ and $\Theta_{\mathcal U}$ are the pointing operators for the vertices and the edges, respectively.
 This translates in terms of generating function as
\[C_{\text{ord}}(z,u) = z + z^2 u \frac {\partial C_{\text{ord}}}{\partial z}(z,u) + z u^3  \frac {\partial C_{\text{ord}}}{\partial u }(z,u)\]
where $C_{\text{ord}}(z,u)$ is the \emph{ordinary} generating function \( \sum_{n,m \geq 1} c_{n,m} z^n u^m \). 
Then the PDE from the statement can be obtained via a Borel transform\footnote{
  The \emph{Borel transform} $\mathcal{B}$ maps ordinary generating functions $f(z) = \sum_{n} f_n z^n$ to exponential generating functions $\left(  \mathcal{B}f \right)(z)  := \sum_{n} f_n \frac{z^n}{n!}$. This operator is linear and has nice properties like $\mathcal B (z f) = \int \mathcal B f$.
  See for example \cite{le_premier_papier_de_matthieu}.
  }
and a differentiation with respect to~$z$.
\end{proof}

We can now tackle the proof of Theorem~\ref{theo:gf}.

\begin{proof} [Proof of Theorem~\ref{theo:gf}] 

  First, we check that the expression of $C(z,u)$ given by Equation~\eqref{eq:Czu} is compatible with the PDE of Lemma~\ref{lem:pde}. 
  Since  
  \[ 
    \pdiff C z (z,u)= e^{\frac{1-u}u}\Cayley\! '\left( \frac{u z + u - 1}{u} e^{\frac{1-u}{u}} \right) 
  \] 
  and 
  \[ \pdiff C u (z,u) = \frac {1-zu}{u^3}e^{\frac{1-u}u}\Cayley\! '\left( \frac{u z + u - 1}{u} e^{\frac{1-u}{u}} \right) - \frac 1 {u^3}, \]
  we have
  \begin{equation}
    (1-zu) \pdiff C z - u^3 \pdiff C u = 1, 
    \label{eq:second_pde}
  \end{equation}
  which is consistent with PDE~\eqref{eq:pde}. Note that we could have replaced $\Cayley$ by any other univariate function in Formula~\eqref{eq:Czu} and it would still satisfy PDE~\eqref{eq:pde}. However, $C(z,u)$ also needs to satisfy the initial condition \(C(0,u) = 0\). To check this, we express $\Cayley$ in terms of the Lambert $W$ function which is implicitly defined by $W(t)\exp(W(t)) = t$
  and satisfies $W(t\exp(t)) = t$.
  More precisely,  formula  \cite[Eq. (2.1)]{lambert-function} translates
 \begin{equation}
  \Cayley (x) = - W(-x) - \frac 1 2 W(-x)^2 
  \label{eq:Cayley_in_terms_of_W}
 \end{equation}
  and hence
  \begin{eqnarray*}
    C(0,u) & = & \Cayley\left(\frac{u-1} u e^{\frac{1-u} u} \right) + \frac 1 {2 u^2} - \frac 1 2 \\ 
    & = & - W\left(\frac{1-u} u e^{\frac{1-u} u} \right) - \frac 1 2 W\left(\frac{1-u} u e^{\frac{1-u} u} \right)^2 + \frac 1 {2 u^2} - \frac 1 2 \\
    & = & \frac {u-1}{u} - \frac{(u - 1)^2} {2 u^2} + \frac 1 {2 u^2} - \frac 1 2 = 0.
  \end{eqnarray*}
    
To conclude, we need to establish the uniqueness of the solution of PDE~\eqref{eq:pde}. It can be done by strong analytic theorems like the Cauchy-Kovalevskaya theorem \cite{dieudonne}. However, we have chosen to roll out the \emph{method of characteristics} \cite{methode-caracteristiques} to show the reader how Formula~\eqref{eq:Czu} can be found. 

Let $C(z,u)$ be any solution of PDE~\eqref{eq:pde}, and set two functions $z(t)$ and $u(t)$ such that 
\begin{equation}
  \left\{
    \begin{array}{rclcl}
  z'(t) & = & 1 -z(t)u(t), &\quad& z(0) = 0 \\
  u'(t) & = & -u(t)^{3}, &\quad& u(0) = 1/\sqrt K,
\end{array}\right. 
\label{eq:system_de}
\end{equation}  
with $K$ be a real positive parameter.
The choice of \eqref{eq:system_de} is justified by the fact that 
\[
  \pdiff \, t  \left( C\left(z(t),u(t)\right) \right)  = (1-z(t)u(t)) \pdiff C z (z(t),u(t)) - u(t)^3 \pdiff C u (z(t),u(t)) = 1
\]
because $C$ satisfies PDE~\eqref{eq:second_pde}. Therefore, this equation can be integrated from 0 to $t$:
\[ C(z(t),u(t)) = C(z(0),u(0)) + t = C(0,1/\sqrt K) + t \]
and hence 
\begin{equation}
  C(z(t),u(t))= t 
  \label{eq:to_be_inversed}
\end{equation} since $C(0,u) = 0$.
We now solve the system of differential equations~\eqref{eq:system_de}. 
First, we see that $-\dfrac{u'}{u^3}$ (which is equal to $1$) is the derivative of $\dfrac 1 {2u^2}$. So by integrating between $0$ and $t$ we see that $\dfrac 1 {2u(t)^2} - \dfrac 1 {2u(0)^2} = t$, and so
\[u(t) = \frac {1} {\sqrt{K + 2t}}.\]
Plugging this into the differential equation for $z(t)$ (see Equation \eqref{eq:system_de}) 
yields a linear ordinary differential equation of order $1$. 
In order to solve it,
multiply both sides by $\exp(\sqrt{K+2t})$, and obtain 
\[e^{\sqrt{K+2t}} z'(t) + \frac 1 {\sqrt{K+2t}}e^{\sqrt{K+2t}}z(t) = e^{\sqrt{K+2t}}.\]
The left-side member is at sight the derivative of \(e^{\sqrt{K+2t}}z(t) \), while the indefinite integral of \(e^{\sqrt{K+2t}}\) is \( \left(  \sqrt{K+2t} - 1 \right) e^{\sqrt{K+2t}}\).
This is why by integration from $0$ to $t$ we get 
\[  e^{\sqrt{K+2t}} z(t) -  e^{\sqrt{K}} z(0) = \left(  \sqrt{K+2t} - 1 \right) e^{\sqrt{K+2t}} - \left(  \sqrt{K} - 1 \right) e^{\sqrt{K}} \]
and since $z(0) = 0$,
\[   z(t) =  \sqrt{K+2t} - 1 - \left(  \sqrt{K} - 1 \right) e^{\sqrt{K} - \sqrt{K+2t} } .\]
Now  that we have the  expressions of $z(t)$ and $u(t)$, we want to solve  $\left\{ \begin{array}{ll}
  z(t) = z \\
  u(t) = u
\end{array} \right.$ in terms of $K$ and $t$.
Plugging $\sqrt{K + 2t} = u^{-1} $ into  the equation $z(t)=z$ gives 
\[ z = u^{-1} - 1 - \left( \sqrt K - 1 \right) e^{\sqrt K - u^{-1} } \]
which can be rewritten as
\[  \left( \sqrt K - 1 \right) e^{\sqrt K - 1} =  \left( - z + u^{-1} - 1 \right) e^{u^{-1} - 1} .  \] 
Using the Lambert W function described above, we then see that
\[ \sqrt K - 1 = W\left(\left( \sqrt K - 1 \right) e^{\sqrt K - 1} \right) = W \left(    \left( - z + u^{-1} - 1 \right) e^{u^{-1} - 1} \right),\]
which implies that
\[ K = W \left(    \left( - z + u^{-1} - 1 \right) e^{u^{-1} - 1} \right)^2 + 2  W \left(    \left( - z + u^{-1} - 1 \right) e^{u^{-1} - 1} \right) + 1. \]
Since $K = \frac 1 {u^2} - 2 t$,
we have 
\[t =  -  \frac 1 2 W \left(    \left( - z + u^{-1} - 1 \right) e^{u^{-1} - 1} \right)^2 -  W \left(    \left( - z + u^{-1} - 1 \right) e^{u^{-1} - 1} \right) + \frac 1 {2u^2} - \frac 1 2.  \]
Using Equation~\eqref{eq:to_be_inversed} and the expression of the $\Cayley$ function in terms of $W$ (see Equation~\eqref{eq:Cayley_in_terms_of_W}), we obtain a formula which is consistent with Equation~\eqref{eq:Czu}.
\end{proof}

\section{Proof by bijection}
\label{s:bij}

This section establishes a direct bijection between increasing 1,2-trees and Cayley trees. Actually, we propose two descriptions of the mapping from 1,2-trees to Cayley trees
(see Figure~\ref{overview} for an overview).
The first one uses three different steps and intermediary structures. It is useful because each step is easily reversible, which sums up into the reciprocal bijection.
The second is condensed in a single step, but it is more difficult to reverse. Though, it is much easier to see in this condensed description a generalization of Shor's decomposition of Cayley trees (as in Figure~\ref{fig:shor-idea}).

\fig{[width=0.95 \textwidth]{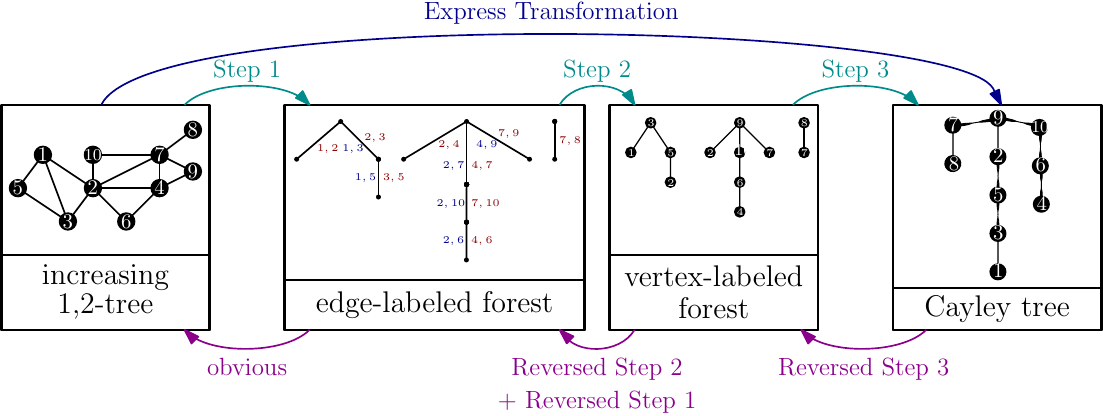}}{  Overview of the bijection from increasing 1,2-trees to Cayley trees. }{overview}

\subsection{From increasing 1,2-trees to Cayley trees}
\label{ss:from-12-to-cayley}

Let $G$ be an increasing 1,2-tree on $n$ vertices, which we are going to transform into a Cayley tree $\tau(G)$.
Remarkably, this direction makes no mention of twists.
The increasing 1,2-tree from Figure~\ref{ex12tree} serves as an example for the construction.

\fig{[scale = 1.4]{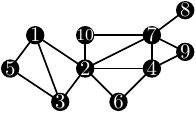}}{An example of an increasing 1,2-tree}{ex12tree}

\paragraph{Step 1: transformation into a plane forest $F$ labeled by edges of $G$.}
This step is illustrated in Table~\ref{table:step1}.

Forest $F$ is plane, meaning that for each vertex, its children are ordered from left to right.
However, the order between the trees of $F$ does not matter.

Each edge of $F$ will carry one or two labels, that we write along the edge on the left or on the right. When there is only one label, it will always be on the right, while when there are two labels, they will be distributed on both sides. These labels correspond to the edges of $G$.

Let us initialize $F$ as an empty forest.   
We successively consider every vertex $\vertex v \neq 1$ of $G$, starting from vertex $\vertex 2$ and ending to vertex $\vertex n$. By construction of an increasing 1,2-tree, vertex $\vertex v$ has $1$ or $2$ neighbors smaller than it.
\begin{itemize}
\item \textbf{Rule 1}. If $\vertex v$ has only one such neighbor $\vertex x$ (which means it was attached to $\vertex x$ as a leaf during the building of $G$), then we add a new rooted tree to $F$, consisting in a sole edge labeled on its right by $(x,v)$.
\item \textbf{Rule 2}. If $\vertex v$ has two neighbors $\vertex x$ and $\vertex y$ such that $x < y < v$ 
  (which means it was attached to edge $\vertex x\vertex y$ as a triangle during the building of $G$),
  then we modify the forest depending on whether the label $(x,y)$ is on the left or right of the edge $e$ that carries it.

\begin{table}
\begin{tabular}{|m{0.15\textwidth}|c|c|}  \firsthline

\begin{center}
\textbf{Rule 1.} \\
$\vertex v$ is attached to a vertex $\vertex x$ as a leaf
\end{center}

 &

\multicolumn{2}{c|}{
 \begin{minipage}{0.75 \textwidth}
\begin{center}
 \includegraphics[scale=1.1]{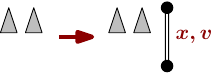}
\end{center}
 \end{minipage}
   }

  \\ \hline

\begin{center}\textbf{Rule 2.} \\
$\vertex v$ is attached to $\vertex x$ and $\vertex y$ as a triangle
\end{center}
&

\begin{minipage}{0.3 \textwidth}
\begin{center}
\vspace{5pt}
the label is on the left \\ \vspace{5pt}
 \includegraphics[width=0.9 \textwidth]{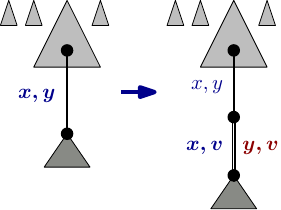}
 \vspace{2pt}
\end{center}
\end{minipage}

 &
\begin{minipage}{0.4 \textwidth}
\begin{center}
the label is on the right \\
 \vspace{5pt}
  \includegraphics[width=0.9 \textwidth]{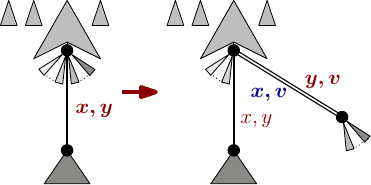}
 \vspace{12pt}
\end{center}
\end{minipage}

  \\ \hline
\end{tabular}

\caption{Three rules to transform an increasing 1,2-tree into a plane forest labeled by edges (Step 1 of the transformation from increasing 1,2-trees to Cayley trees)\label{table:step1}}
\end{table}

\begin{itemize}
\item \textbf{Rule 2 Left}. If $(x,y)$ is on the left of $e$, then we split $e$ into two edges $e$ and $e'$, with $e$ above $e'$
(see Table~\ref{table:step1} bottom left).
Doing so, every subtree incident and below $e$ becomes incident and below $e'$ instead.
We put the label $(x,v)$ on the left of $e'$, and the label $(y,v)$ on the right of $e'$.
The labels on $e$ remain the same.
\item  \textbf{Rule 2 Right}. If $(x,y)$ is on the right of $e$, then we insert a new edge $e'$ in $F$ as the right sibling of $e$, pushing below $e'$ every subtree that shares an endpoint with $e$ and is on the right of $e$ (see Table~\ref{table:step1} bottom right). We put a label $(x,v)$ on the left of $e'$, and a label $(y,v)$ on the right of $e'$.
\end{itemize}
\end{itemize}

A full example for the graph of Figure~\ref{ex12tree} is detailed at Table~\ref{table:step1-detailed}.

\begin{table}[htp]

\begin{center}
\newcounter{ct}
\begin{tabular}{|m{42pt}|c|c|c|c|c|c|c|}
\hline
vertex $v$ &\setcounter{ct}{2}
\whiledo {\value{ct} < 6}
{\thect
\if \thect 5 \else \esperluette \fi
\stepcounter {ct}
  }
\\ \hline
forest $F$ &
 \setcounter{ct}{2}
\whiledo {\value{ct} < 6}
{
$\vcenter{\hbox{\includegraphics{images/step1v\thect}}}$
\if \thect 5 \else \esperluette  \fi
\stepcounter {ct}}
  \\ \hline
vertex $v$ &

\multicolumn{2}{|c|}{
6 }

&

\multicolumn{2}{|c|}{
7 }

  \\ \hline
 forest $F$ &
  \multicolumn{2}{|c|}{
 $\vcenter{\hbox{\includegraphics{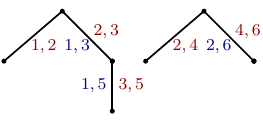}}}$  }
&
  \multicolumn{2}{|c|}{
 $\vcenter{\hbox{\includegraphics{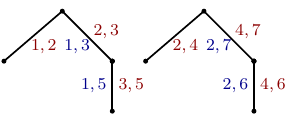}}}$  }

\\
  \hline

  vertex $v$ &

\multicolumn{2}{|c|}{
8 }

&

\multicolumn{2}{|c|}{
9 }

  \\ \hline
 forest $F$ &
  \multicolumn{2}{|c|}{
 $\vcenter{\hbox{\includegraphics{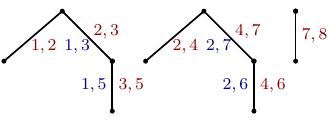}}}$  }
&
  \multicolumn{2}{|c|}{
 $\vcenter{\hbox{\includegraphics{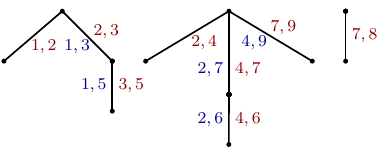}}}$  }

\\
  \hline

    vertex $v$ &

\multicolumn{4}{|c|}{
10 }

  \\ \hline
 forest $F$ &
  \multicolumn{4}{|c|}{
 $\vcenter{\hbox{\includegraphics{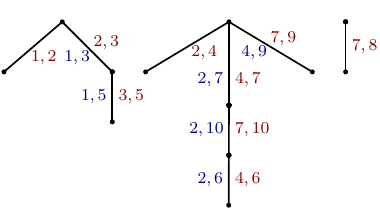}}}$  }

\\
  \hline
\end{tabular}
\end{center}
\caption{Construction of the forest $F$ edge by edge, starting from the 1,2-tree of Figure~\ref{ex12tree} }
\label{table:step1-detailed}
\end{table}


\begin{definition}[root edge]
  The \emph{root edge} of a tree is the leftmost edge incident to the root.
 \end{definition}

\begin{lemma} Let $t$ be a tree of $F$. At the end of Step 1, the root edge of $t$ leads to a leaf and does not carry a left label, but a right label of the form $(a,b)$, where $a$ is the smallest integer appearing in $t$.
\label{lem:left_leaf}
\end{lemma}

\begin{proof} In order to create the tree $t$, Rule 1 must have been applied. At this moment, $t$ consists in a single edge with a right label of the form $(a,b)$ with $a < b$. The property holds at this point and we prove that it continues at each substep of Step 1 by induction.

Let us suppose that an edge $e'$ is added to $F$.
If $e'$ is in a different tree than $t$, then it is clear that the property is still satisfied.

Assume that $e'$ is added to $t$. It must come with Rule 2 Left or Rule 2 Right. Let $e$ be the edge on which Rule 2 has been applied.

If $e$ is not the root edge of $t$, then the latter is still a leaf and keeps the same right label.

Let us assume now that $e$ is the root edge of $t$. Since by induction $e$ does not have a left label, Rule 2 Right must have been applied. A new right sibling is added to $e$, but nothing below $e$. The edge $e$ remains a leaf. 

Finally, the integers $x,y,v$ appearing in the left label $(x,v)$ and right label $(y,v)$ of $e'$ must be at least equal to $a$ 
since $v$ is the largest integer in the forest so far, 
and $x$ and $y$ were already in $t$ before insertion of $e'$
 (and so must be at least $a$ by induction).
\end{proof}

In the following, we use the following definition.

\begin{definition}[end label]
 The \emph{end label} of an edge $e$ with right label $(x,y)$ is $y$.
\end{definition}

Note that if $e$ also has a left label $(x',y')$, then by construction $y=y'$.

\paragraph{Step 2: label the vertices of $\boldsymbol F$.}

Let $t$ be any tree of $F$.  We are going to label the vertices of $t$ by integers between $1$ and $|G|$.

First, consider the root edge of $t$ and its right label $(x,y)$. We label the endpoint of the root edge (which is a leaf by Lemma~\ref{lem:left_leaf}) with $x$.


Next, we proceed to a Depth First Search of $t$ where one favors the rightmost edge at each step. The visit order of this DFS gives an ordering for the vertices: ${v_1},\dots,{v_{|t|}}$, and for the edges: $e_1, \dots, e_{|t| - 1}$. An example is shown in Figure~\ref{bfs}.

\fig{[width=0.3 \textwidth]{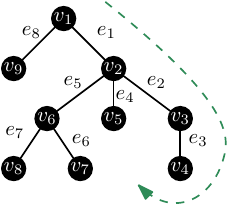}}{Example of a rightmost DFS of a tree.}{bfs}

Then, for $k$ between $1$ and $|t| - 1$, we label ${v_k}$ by the {end label} of $e_k$.
Vertex ${v_{|t|}}$ is the endpoint of the root edge of $t$, which is already labeled.

An example of this step is shown at Figure~\ref{step2}.

\fig{[scale = 1.2]{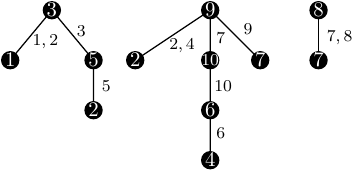}}{  Result of Step $2$ for the increasing 1,2-tree of Figure~\ref{ex12tree}.
We have kept the end label of all edges and the whole right label of the root edge to illustrate how they are used to label vertices.}{step2}

\paragraph{Step 3: merge the trees of $\boldsymbol F$.}

Remove any edge label from $F$ and forget the fact that $F$ is plane (i.e. ordered). Merge every vertex of $F$ that carries the same vertex. There results a Cayley tree, which we denote $\tau(G)$. Following our example, the result is depicted at Figure~\ref{excayley}.

\fig{[scale = 1.2]{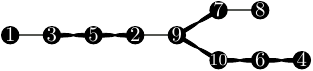}}{Cayley Tree in bijection with the 1-2 increasing tree of Figure~\ref{ex12tree}. }{excayley}

\subsection{Properties of \taut}

We now prove several properties on the map $\tau$ defined in the previous subsection. These properties will be useful to describe the reciprocal of $\tau$. Incidentally we show the following proposition, which induces Theorem~\ref{theo:central} provided that $\tau$ is a bijection.

\begin{proposition}
The map $\tau$ sends increasing 1,2-trees with $n$ vertices and $m$ triangles to Cayley trees with $n$ vertices and $m$ twists. More specifically, $\tau$ 
 transforms each 2-connected component of an increasing 1,2-tree into a subtree where every edge but one is a twist.

In particular, increasing trees (i.e. increasing 1,2-trees where we only attach new vertices as leaves) are preserved by $\tau$, 
and increasing 2-trees (i.e. increasing 1,2-trees where we only attach new vertices as triangles) are sent to Cayley trees where every edge is a twist,
except for one which must be the only edge leaving $\vertex 1$.
\label{prop:tau}
\end{proposition}

We need beforehand the following definition.

\begin{definition}[\textsc{Eddy}\footnote{The name is a homage to the famous French singer Eddy Mitchell \url{https://en.wikipedia.org/wiki/Eddy_Mitchell}.}]
Let $\vertex v$ be a node of a plane forest, not incident to the root edge of a tree. Consider $\vertex a$ the youngest ascendant of $\vertex v$ ($\vertex v$ included) having left siblings. We define the \textsc{Eddy} of $\vertex v$, denoted by $\eddy{\vertex v}$, as the edge linking the sibling just on the left of $\vertex a$ to its parent.
By extension, for an edge $\vertex u \vertex v$ where $\vertex v$ is the child of $\vertex u$, $\eddy{\vertex u \vertex v}$ denotes $\eddy{\vertex v}$.
\end{definition}

\fig{[scale = 1]{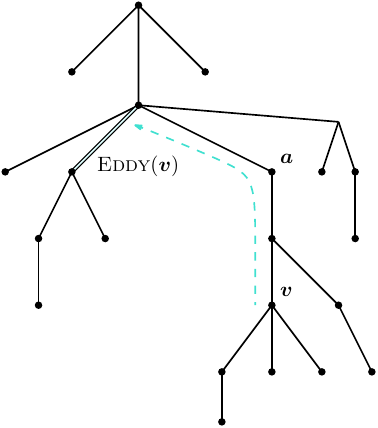}}{Example of $\eddy{\vertex v}$, given a vertex $\vertex v$ from a tree}{eddy}

This definition is illustrated by Figure~\ref{eddy}. Note that all the vertices in the leftmost branch starting at $a$ share the same $\textsc{Eddy}$.
Note also that during a rightmost DFS, the edge visited immediately after visiting
any edge $e$ leading to a leaf is $\eddy{e}$.

By a careful analysis of the rules of Step~1, we infer the following observation.

\begin{lemma}\label{lem:edyfinitif}
 For every non-root edge $e$ in the forest of Step~1,
the labels of $\eddy{e}$ do not vary during the whole construction of the forest.
\end{lemma}

 \begin{proof}

 Let $e = \vertex u \vertex v$ be a non-root edge of the forest and let $f$ be an edge added to the forest after using any rule from Step~1. We show that the insertion of $f$ does not change the labels of $\eddy{e}$.
 
 Let $\vertex a$ be the youngest ascendant of $\vertex v$ with left siblings before insertion of $f$.
 In any of the three rules, $f$ has no right sibling. 
 Thus the insertion of $f$ does not add a left sibling to an existing vertex. Therefore the youngest ascendant of $\vertex v$ with left siblings after insertion of $f$ is still $\vertex a$, or an ascendant of $\vertex a$.

%
%
%

If it is still $\vertex a$, then the edge on the left of $\vertex a$ is the same before and after the insertion of $f$ (because $f$ cannot be a left sibling of another edge). So $\eddy e$ remains the same.

If the youngest ascendant of $\vertex v$ with left siblings is not $\vertex a$ anymore, this means that $\vertex a$ has lost its left siblings. So $f$ must be the new edge between $\vertex a$ and its parent. Since the insertion of $f$ can be reverted by contracting $f$, the new parent of $\vertex a$ must have the same left sibling than $\vertex a$ before insertion of $f$. So the youngest ascendant of $\vertex v$ with left siblings is now the parent of $\vertex a$, and $\eddy e$ is still the same edge. 
 \end{proof}

From this, we can deduce the following.

\begin{lemma}\label{lem:edyminue}
 For every non-root edge $e$ in the forest of Step~1, the end label of $\eddy e$ is less than the end label of $e$.
\end{lemma}

\begin{proof}
 For an edge $e$, $\eddy e$ was already in the forest when $e$ is introduced.  So the end label of  $\eddy e$ is necessarily less than the end label of $e$ (since the edges are introduced in increasing order of their end label). The property remains true by Lemma~\ref{lem:edyfinitif}.
\end{proof}

\begin{lemma}\label{lem:increasing-tout}
Given a vertex $\vertex u$ of the forest  of Step~1 with children $\vertex{v_1},\dots,\vertex{v_k}$, the end labels of all edges in the subtrees rooted at $v_{i+1}, \dots, v_k$ are greater than the end label of the edge $uv_i$.
Moreover, the end labels of the edges going from $u$ to its children are increasingly ordered from left to right.
\end{lemma}

\begin{proof}
The property directly comes from one or several uses of Lemma~\ref{lem:edyminue}.
\end{proof}

In the following, we use the natural extension of the definition of twists (Definition~\ref{def:twist}) to rooted subtrees of Cayley trees (i.e. rooted trees whose set of labels is not necessarily $\{1,\dots,n\}$).

\begin{lemma} At the end of Step 2, every edge of every tree is a twist.
\label{lem:totally-twisted}
\end{lemma}

\begin{proof} Before commencing the proof, remark the two following points:

1. By construction (application of Step 2), the label of a leaf $\ell$ is the end label of $\eddy{\ell}$.

2. Also by construction, a non-leaf vertex inherits the end label of the edge towards its rightmost child.

Now consider an edge $e = \vertex u \vertex v$ (where $\vertex v$ is the child of $\vertex u$). 

If $e$ is the root edge of its tree, then $\vertex v$ is a leaf (Lemma~\ref{lem:left_leaf}) which inherits (by construction) the smallest label of the tree. Therefore $e$ is a twist.

Now let us assume that $e$ is not a root edge, hence $\eddy e$ is defined. Let $\vertex \ell$ be the leftmost leaf descendant of $\vertex v$. We are going to prove that the label of $\vertex u$ is greater than the label of $\vertex \ell$.
To do so, we introduce $v'$, the rightmost child of $\vertex u$ (which may or may not be $v$).

We have the chain of inequalities:
\begin{align*}
 u & = \text{end label}(uv') & \text{(by point 2)} \\
  & \geq \text{end label}(e) & \text{(by Lemma~\ref{lem:increasing-tout})} \\
  & > \text{end label}(\eddy e) & \text{(by Lemma~\ref{lem:edyminue})} \\
  & = \text{end label}(\eddy \ell) & \text{(since $\eddy e = \eddy \ell$ by definition of $\ell$)} \\
  & = \ell & \text{(by point 1).} 
\end{align*}
This implies that $e$ is a twist.
\end{proof}

We can now show the proposition stated at the beginning of this subsection.

\begin{proof}[Proof of Proposition~\ref{prop:tau}.] Let $G$ be an increasing 1,2-tree.
We can recover $\tau(G)$ by merging the trees of the forest at the end of Step~2 in a specific order.
By iterating over all trees $t$ of $F$ with respect to their minimal label, we amalgamate $t$ with the previously joined trees by merging the smallest vertex of $t$ with the other vertex with the same label. By construction and by  Lemma~\ref{lem:left_leaf}, the smallest vertex of $t$ is a leaf incident to the root edge of $t$. So if $e$ is the root edge of $t$, then $e$ is increasing in $\tau(G)$. By Lemma~\ref{lem:totally-twisted}, all other edges of $t$ are twists.

Thus the final number of twists in $\tau(G)$ is equal to the number of non-root edges at the end of Step~1, which is also the number of times where Rule 2 has been used, and so the number of triangles in $G$.

To prove the more specific statement of the proposition, note that  the $2$-connected components of $G$ are in correspondence with trees of $F$. Indeed, all the edges of a $2-$connected component appear as left and right labels of the corresponding tree. This can be proved by induction.
Since all the edges of a tree $t$ of $F$ are twists at the end of Step 2 (by Lemma~\ref{lem:totally-twisted}), $t$ forms at the end of Step~3 a subtree of $\tau(G)$ where every edge is a twist, except the root edge.
\end{proof}

Before considering the reverse direction, let us establish the following properties that will prove useful.

\begin{lemma} \label{lem:smallest}
  Let $T$ be a tree of $F$ and $S$ be any subtree of $T$. At the end of Step 2, the vertex with the smallest label in~$S$ is the leftmost leaf of $S$.
  \end{lemma}

  \begin{proof} 

  Let $\vertex v$ be the root of $S$, and $\vertex \ell$ be the leftmost leaf of $S$.

  If  $\vertex v$ is also the root of $T$, then the lemma obviously holds by Lemma~\ref{lem:left_leaf} ($\vertex \ell$ is the leftmost child of $\vertex v$ and carries the smallest vertex of $T$).
  
  So we can suppose that $\vertex v$ has a parent, say $\vertex u$. As mentioned by point 1 of the previous proof, $\ell$ is the end label of $\eddy{\ell}$, which is the same as $\eddy{\vertex u \vertex v}$. By Lemma~\ref{lem:increasing-tout}, the end label of $\eddy{\vertex u \vertex v}$ must be smaller than the end label of any edge of $S$. So $\ell$ is smaller than the end label of any edge of $S$.

  Let $\vertex x$ be a vertex of $S$ different from $\vertex \ell$. 
   Since the vertex labels are induced by a rightmost DFS, the label of $\vertex x$ is necessarily an end label of an edge of $S$. According to what we said above, the label of $\vertex x$ is greater than the label of  $\vertex \ell$.
  \end{proof}

\begin{lemma} \label{lem:ordering}
  Let $\vertex u$ be a vertex of $F$. At the end of Step 2, the subtrees rooted in $\vertex u$ are sorted from left to right with respect to their minimum.
  \end{lemma}
  
  \begin{proof}
    Let $\vertex{v_1},\dots,\vertex{v_k}$ be the children of $\vertex u$, in this order, and $T_1,\dots,T_k$ the corresponding subtrees. Let us fix $i$ and $j$ such that $ 1 \leq i < j \leq k$.
    
    If $i = 1$, then by Lemma~\ref{lem:smallest}, the minimum of $T_1$ is the leftmost leaf of $T_1$, which is also the leftmost leaf of the subtree rooted at $\vertex u$, and so the minimum of the whole subtree. 
    In particular, the minimum of $T_1$ is smaller than the minimum of $T_j$.

    If $i > 1$, then the minimum of $T_i$ and the minimum of $T_j$ are given by the leftmost leaves of $T_i$ and $T_j$ (again by Lemma~\ref{lem:smallest}).
    But these leaves are labeled with the end labels of $\vertex {u v_{i-1}}$ and $\vertex {u v_{j-1}}$, respectively. 
    By Lemma~\ref{lem:increasing-tout}, the former is smaller than the latter, which concludes the proof.  \end{proof}

\subsection{From Cayley trees to increasing 1,2-trees}
\label{ss:from_cayley_to_increasing}

In this subsection, we describe how to recover the increasing 1,2-tree $\tau^{-1}(T)$, given a Cayley tree $T$. The process will be done by reversing the 3 steps from Subsection~\ref{ss:from-12-to-cayley}. Remember that $T$ is rooted at $\vertex 1$.

\paragraph{Reversed Step 3: cutting $T$ into a forest and order the children at each node.} This step is illustrated by Figure~\ref{reversed_step3}.

For each increasing edge $\vertex u \vertex v$ of $T$ where $\vertex v$ is the child of $\vertex u$,
split $T$ in two  by detaching edge $\vertex u \vertex v$  from $\vertex u$
(vertex $\vertex u$ is duplicated).
We root the new tree at vertex $\vertex v$.
Since $\vertex 1$ is only linked to the rest of the tree by increasing edges, the process must induce a tree with a single vertex $\vertex 1$. Discard this tree.

Once done, order the children at each node so that the subtrees are ordered with respect to their minimal vertex (see Definition~\ref{def:min}).
Note that the leftmost child of the root at each tree is a leaf: since we have cut the tree at each increasing edge $\vertex u \vertex v$, the minimal vertex of the subtree is $\vertex u$, which must be the leftmost child of the root after reordering the children.

\fig{[width = 0.75 \textwidth]{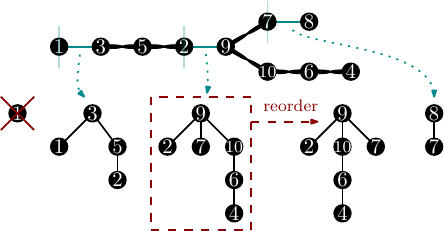}}{Example of a Reversed Step 3, given the Cayley tree of Figure~\ref{excayley}.
The ordering of the children has changed in only one tree in the forest,
since the other trees were already rightly ordered. }{reversed_step3}

\paragraph{Reversed Step 2: computing the end-label of the edges}

Let $t$ be any tree of $F$.  We are going to add  left labels and right labels of the form $(?,v)$ to edges of $t$, where the question mark $?$ is a temporary symbol which will be replaced by a vertex label in the next (reversed) step, and $v$ is the vertex label.

\fig{[width = 0.7 \textwidth]{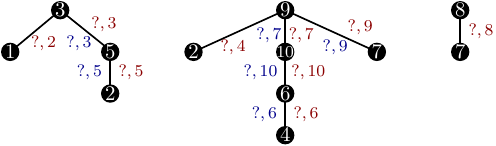}}{Reversed Step 2 with the example of the Cayley tree of Figure~\ref{excayley}.  }{reversed_step2}

To do so, proceed to a Depth First Search of $t$ where one favors the rightmost edge at each step, and name vertices $\vertex{v_1},\dots,\vertex{v_{|t|}}$, and edges $e_1, \dots, e_{|t| - 1}$ with respect to the visit order in the DFS (see again Figure~\ref{bfs}).

Then, for $k$ between $1$ and $|t| - 2$, we add $(?,v_k)$ as the left \emph{and} right labels of $e_k$. For $k = |t|-1$, we just add a right label $(?,v_{|t|-1})$ (no left label) at $e_{|t|-1}$.

This step is shown at Figure~\ref{reversed_step2} with the same example as earlier.

\paragraph{Reversed Step 1: fully identifying the 1,2-tree edges.} We have now to fill in the question marks of the left and right labels to be consistent with the rules of Step 1 (see Table~\ref{table:step1}).

\begin{table}[h!]
\begin{center}
\begin{tabular}{|m{25pt}|c|c|c|c|c|c|c|}
\hline
edge & 
initial state &
$9 - 2$ &
 $9 - 10$
\\ \hline
tree &
 \setcounter{ct}{0}
\whiledo {\value{ct} < 3}
{
$\vcenter{\hbox{\includegraphics[scale=1.1]{images/reversed_step1v\thect}}}$
\if \thect 2 \else \esperluette  \fi
\stepcounter {ct}}
  \\ \hline
edge & 
$ 10 - 6$ &
$6 - 4$ &
$9 - 7$
\\ \hline
tree &
 \setcounter{ct}{3}
\whiledo {\value{ct} < 6}
{
$\vcenter{\hbox{\includegraphics[scale=1.1]{images/reversed_step1v\thect}}}$
\if \thect 5 \else \esperluette  \fi
\stepcounter {ct}}
  \\ \hline

 \end{tabular}
 \end{center}

\caption{Reversed Step 1, edge by edge, given the second tree of Figure~\ref{reversed_step2}.}
\label{tab:reversed_step_1}
\end{table}

Let $t$ be any tree of $F$.  Let $\vertex \ell$ be the leftmost child of the root of $t$. As we noticed at the end of Reversed Step 3, $\vertex \ell$ is a leaf.
Since the edge linking the root of $t$ to $\vertex \ell$ has been visited last in the rightmost DFS, it has only a right label, which is of the form $(?,y)$.
Replace this right label by $(\ell,y)$.

Then, iterate over all edges $e$ of $t$ via a Depth First Search, which this time favors the leftmost edges. We skip the first edge, which connects the root to $\vertex \ell$ (this edge has already been treated).

\fig{[scale = 1.2]{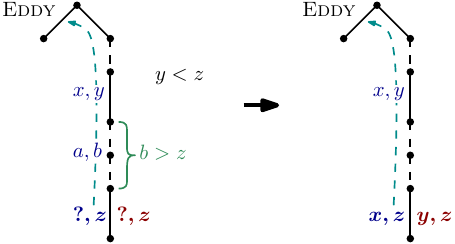}}{How to recover the vertices behind the question marks. }{recovering_edge}

Let $(?,z)$ be the left label (or right label) of $e$.
By induction, the question marks appearing in the left and right labels of every edge linking $e$ to $\eddy{e}$ have been replaced.
Among all of these edges, $\eddy{e}$ excluded, 
we pick the first one, starting from the bottom, 
which has a left label $(x,y)$ such that $y < z$. 
If no such edge exists, then we set $(x,y)$ as the right label of $\eddy{e}$. 
We then replace the left label of $e$ by $(x,z)$ and the right label of $e$ by $(y,z)$.
This process is schematized by Figure~\ref{recovering_edge}.

A full example is given by Table \ref{tab:reversed_step_1}.

\begin{proposition}
The above transformation describes the inverse of $\tau$. In particular, $\tau$ is a bijection.
\end{proposition}

\begin{proof}
  \textbf{Reversed Step 3 is the inverse of Step 3.} It is not complicated to see that after Reversed Step 3, we can recover the original Cayley tree by merging the trees of the forest as in Step~3.
  
  We also need to prove that the decomposition of Reversed Step 3 is unique,
  i.e. there is no other way to cut a Cayley tree that would form a forest in the image of Step 2.
  
  First, by Lemma~\ref{lem:left_leaf}, the leftmost child of the root of each tree must be a leaf and carry the smallest label in its tree. 
  It implies that when we attach the trees back together, the edges incident to the points of attachment are necessarily increasing, which is the case in Reversed Step 3.
  
  Second, by Lemma~\ref{lem:totally-twisted}, every edge of each tree must be a twist. This is checked here since during Reversed Step 3, we have cut all the increasing edges.

  Finally, the ordering of the children is uniquely determined by the minima, as stated by Lemma~\ref{lem:ordering}. 

  \textbf{Reversed Step 2 is the inverse of Step 2.}
  In both Step 2 and Reversed Step 2,
  we match the edge end labels with the vertex labels by respecting the visit order of a rightmost DFS. During Step 2, we label the vertices in Step 2, while in Reversed Step 2, we label the edges.

  \textbf{Reversed Step 1 is the inverse of Step 1.} The property given by Fi\-gure~\ref{recovering_edge}, 
  namely 
  "\textit{If $e$ is an edge with left label $(x,z)$  and right label $(y,z)$, then $(x,y)$  is the first label encountered along the path from $e$ (left side) to $\eddy e$ (right side), which satisfies $y < z$.}", always holds during Step~1. 
  Indeed, this is true when we insert $e$ in the forest (see Rule~2 Table~\ref{table:step1} for evidence). It remains true by induction since every edge we insert afterwards has a left label $(a,b)$ where $b > z$.

  Moreover, this property uniquely characterizes the numbers to be used in place of question marks.
\end{proof}

\subsection{Express transformation}

\begin{table}[htb]
\begin{tabular}{|m{0.15\textwidth}|c|c|}  \firsthline

\begin{center}
\textbf{Rule I.} \\
$\vertex v$ is attached to a vertex $\vertex x$ as a leaf
\end{center}

 &

\multicolumn{2}{c|}{
 \begin{minipage}{0.75 \textwidth}
 \hspace{0.25cm}
\begin{center}
 \includegraphics[scale=1.1]{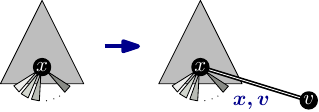}
\end{center}
\hspace{0.25cm}
 \end{minipage}
    }

  \\ \hline

\begin{center}\textbf{Rule II.} \\
$\vertex v$ is attached to $\vertex x$ and $\vertex y$ as a triangle
\end{center}
&

\begin{minipage}{0.3 \textwidth}
\begin{center}
\vspace{5pt}
the label is on the left \\ \vspace{5pt}
 \includegraphics[width=0.95\textwidth]{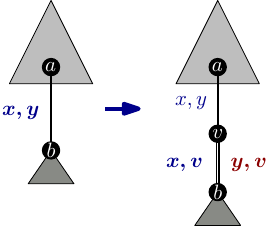}
\vspace{5pt}
\end{center}
\end{minipage}

 &
\begin{minipage}{0.4 \textwidth}
\begin{center}
the label is on the right \\
 \vspace{5pt}
  \includegraphics[width=0.95\textwidth]{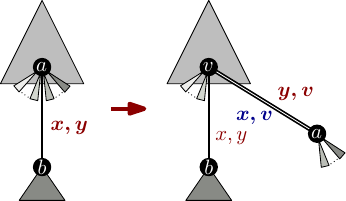}
 \vspace{15pt}
\end{center}
\end{minipage}

  \\ \hline
\end{tabular}

\caption{Three rules of the express transformation\label{table:express}.}
\end{table}

It is possible to express $\tau$ in a more straightforward manner, skipping the Depth First Search of Step~2 by incorporating the vertex labels directly in Step~1. The transformation is summed up in Table~\ref{table:express}.

More precisely, given an increasing 1,2-tree $G$, we construct a plane tree $\mathbb T$ where the edges carry left or/and right labels, like in Step~1 in Subsection~\ref{ss:from-12-to-cayley}, but on top of that, the vertices are labeled by a number between $1$ and $n$.

Let us initialize $\mathbb T$ as a tree with only one vertex labeled by $1$.
Like Step~1 of $\tau$,
 we successively consider every vertex $\vertex v \neq \vertex 1$ of $G$,
  starting from vertex $\vertex 2$ and ending to vertex $\vertex n$.
\begin{itemize}
\item \textbf{Rule I}. If $\vertex v$ has only one neighbor $\vertex x$ such that $x < v$ (which means it was attached to $\vertex x$ as a leaf during the building of $G$), then we add a new leaf $\vertex v$ in $\mathbb T$ as the rightmost child of vertex $\vertex x$. The edge linking $\vertex x$ to $\vertex v$ carries a left label $(x,v)$.
\item \textbf{Rule II}. If $\vertex v$ has two neighbors $\vertex x$ and $\vertex y$ such that $x < y < v$ (which means it was attached to edge $\{\vertex x,\vertex y\}$ as a triangle during the building of $G$), then we modify $\mathbb T$ depending on whether the label $(x,y)$ is on the left or right of the edge $e = \vertex a \vertex b$ that carries it.

\begin{itemize}
\item \textbf{Rule II Left}. If $(x,y)$ is on the left of $e$, then we add a new vertex $\vertex v$ in the middle of $e$
(see Table~\ref{table:express} bottom left).
The edge $\vertex a \vertex v$ inherits the left label and the right label (if any) of $e$. 
We put a label $(x,v)$ on the left of $\vertex a \vertex v$, and a label $(y,v)$ on the right of $\vertex a \vertex v$.
\item  \textbf{Rule II Right}. If $(x,y)$ is on the right of $e$, then we change the label of $\vertex a$ by $v$ and we insert a new edge $\vertex v \vertex a$ in $\mathbb T$ as the right sibling of $\vertex v \vertex b$, pushing below $\vertex v \vertex a$ every subtree that was incident to $e$ and on the right of $e$ (see Table~\ref{table:express} bottom right). We put a label $(x,v)$ on the left of $\vertex v \vertex a$, and a label $(y,v)$ on the right of $\vertex v \vertex a$.
\end{itemize}
\end{itemize}

The process is illustrated step by step by Table~\ref{table:express-detailed}.

Note that this construction is reminiscent of Shor's decomposition of Cayley tree with respect to their twists (cf Figure~\ref{fig:shor-idea}). Rule II Left can be seen as the special case $a=0$ of Figure~\ref{fig:shor-idea} bottom, while Rule II Right represents the case $a>0$.

\begin{table}[hp]
\begin{center}
\newcounter{cct}
\begin{tabular}{|m{42pt}|c|c|c|c|c|c|}
\hline
vertex $v$ &\setcounter{cct}{1}
\whiledo {\value{cct} < 7}
{\thecct
\if \thecct 6 \else \esperluette \fi
\stepcounter {cct}
  }
\\ \hline
\begin{center}
tree
\end{center}&
 \setcounter{cct}{1}
\whiledo {\value{cct} < 7}
{$\vcenter{\hbox{ \includegraphics{images/estep1v\thecct}}}$
\if \thecct 6 \else \esperluette  \fi
\stepcounter {cct}}

\\
\hline
vertex $v$ &
7 & 8 &
  \multicolumn{2}{|c|}{9}
&
\multicolumn{2}{|c|}{10}
\\ \hline
\begin{center}
tree
\end{center}
 \setcounter{cct}{7}
\whiledo {\value{cct} < 9}
{ \esperluette
$\vcenter{\hbox{\includegraphics{images/estep1v\thecct}}}$
\stepcounter {cct}}

& \multicolumn{2}{|c|}{$\vcenter{\hbox{\includegraphics{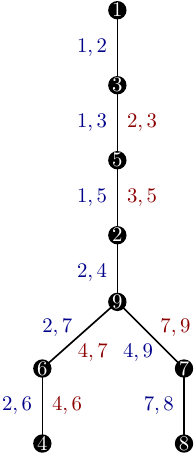}}}$ }
& \multicolumn{2}{|c|}{$\vcenter{\hbox{\includegraphics{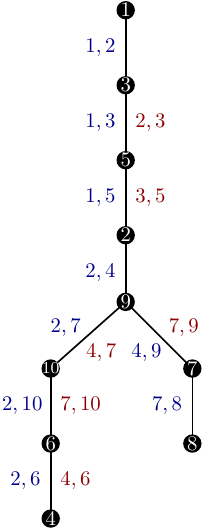}}}$ }

\\
  \hline
\end{tabular}
\end{center}
\caption{Construction of the tree through the express transformation from the increasing 1,2-tree of Figure~\ref{ex12tree} }
\label{table:express-detailed}
\end{table}

\begin{proposition}
If we remove the edge labels from $\mathbb T$ and 
forget the planarity (i.e how children are ordered), then we recover $\tau(G)$.
\end{proposition}

\begin{proof}[Proof ideas.] Instead of a formal dragging on proof, the authors prefer to explain in rough words why the two transformations are identical. 
We introduce a new intermediary rule, named Rule 1', 
to justify why Rule 1 and Rule I induce the same Cayley tree.
Rule 1' is a modification of Rule 1 from Table~\ref{table:step1} 
where the new edge carries a left label instead of a right label (cf. Table~\ref{table:Rule1'}).\footnote{
Even if it does not seem consistent, 
there are reasons why in $\tau$ we use right labels (Rule 1) and for the express transformation we use left labels (Rule I).  
On one hand, using right labels in $\tau$ makes the rightmost Depth First Search of Step 2 more natural (we need to visit the root edge last).
On the other hand, if we had used right labels in the express transformation, then it would be possible that a vertex label will be shifted in a wrong place by using a Rule II Right on a root edge.
}

\begin{table}[htb!]
\begin{tabular}{|c|c|c|}  \firsthline

\textbf{Rule 1}
&
\textbf{Rule 1'}
 &
\textbf{Rule I}
\\
 \includegraphics[scale=1.1]{images/leaf_rule}
 &
 \includegraphics[scale=1.1]{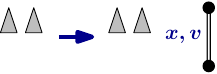}&
  \includegraphics[scale=1.1]{images/leaf_rule_express}

  \\ \hline
\end{tabular}
\caption{Introduction of Rule 1'.}
  \label{table:Rule1'}

\end{table}

\paragraph{Claim 1. Rule 1 is equivalent to Rule 1'.}
Let us suppose that we modify Rule 1 in the original transformation $\tau$ 
so that the edge in the new tree carries a left label and no right label. 
Then, instead of trees such that the leftmost children incident to the root are leaves (see Lemma~\ref{lem:left_leaf}),
 we obtain trees where the root has only one child
 (the proof is similar as the one of Lemma~\ref{lem:left_leaf}). 
These two families of trees are in bijection: to recover the original trees,
we have to topple the root edges to the right so that they become leaves at the far left of the new root,
and change the left labels into right labels.

\paragraph{Claim 2. Rule I (where we attach a leaf to the same unique tree) is equivalent to Rule~1' (where we add a new one-edge tree to the previous forest).}
In comparison with $\tau$, the express transformation does not treat the 2-connected components separately.
 In order to recover the trees of Step~1 of $\tau$, we can cut $\mathbb T$ at the root of the edges with no right label.

Note also that the sources\footnote{A \emph{source} of an edge $\vertex u \vertex v$ is the vertex $\vertex u$.} 
of the edges with no right label keep the same vertex labels throughout the construction of $\mathbb T$. 
Indeed, we can prove that the edges with no right label are always the rightmost children of their parents, which ensures that we cannot apply Rule II Right to one of their right siblings,
which would change the label of the source.
Thus, if we cut $\mathbb T$ at the root of the edges with no right label at any point during the construction, 
we can rebuilt $\tau$ by merging vertices with the same label, 
like in Step~3 of $\tau$.

\paragraph{Claim 3. The vertex labeling of $\mathbb T$ is consistent with a rightmost Depth First Search (Step 2 of $\tau$).} Cut $\mathbb T$ at the root of each edge $\vertex u \vertex v$ with no right label, and start a rightmost DFS starting at $\vertex v$. We claim that for each integer $i$, the label of the $i$-th visited vertex is equal to the end label of the $i$-th visited edge, as in Step~2 of $\tau$.

This can be proved by induction and by thoroughly checking how the vertex labels evolve in Rule~II. For example, in Rule II Right (Table~\ref{table:express} bottom right), $v$ is the label of the newly inserted $\vertex v$ and the end label of the rightmost incident edge of $\vertex v$, which agrees with the visit order of a rightmost DFS. In the same vein, we have to check that $y$, the end label of the edge $\vertex v \vertex b$, corresponds to the label of the leftmost leaf descendant of $\vertex a$ (which is true by induction). 
\bigskip

By combining all three claims, which together factor all the differences between $\tau$ and the construction of $\mathbb T$, we infer that the two constructions lead to the same Cayley tree. 
\end{proof}

%
%
%
%

\section{Uniform random sampling in linear time}
\label{s:random_generation}

In this section, we describe two sampling algorithms which given an integer $n$, output uniformly at random:
\begin{itemize}
\item a Cayley tree with $n$ vertices,
\item an increasing 1,2-tree with $n$ vertices.
\end{itemize}
The computation time is guaranteed in $O(n)$ for both algorithms, even in the worst-case scenario. 

All the complexities in this section do not include the cost of a Random Number Generator. Indeed, we assume that drawing a random integer between $0$ and $n$ is in~$O(1)$ and not in $O(\log(n))$.

\subsection{Random Cayley trees}

Most random samplers for Cayley trees found in the literature or on the Internet are based on the bijection between Cayley trees and Prüfer codes~\cite{prufer}.
The idea is to draw a uniform Prüfer code, that is a sequence of $n-2$ numbers between $0$ and $n-1$, 
and use the Prüfer transformation to get a uniform random tree.

The book by Nijenhuis and Wilf~\cite{nijenhuis-wilf} already introduced this idea by presenting a random sampler running in $O(n^2)$. This complexity can be reduced to $O(n)$ using the algorithm described in~\cite{optimal_prufer}.

We propose here a new elementary sampling algorithm running in $O(n)$, whose idea is directly inspired from the Aldous-Broder algorithm \cite{aldous}.
The Aldous-Broder algorithm is a random walk-based algorithm that generates a uniform random spanning tree of any connected undirected graph $G$. The algorithm starts from an arbitrary vertex and proceeds by randomly traversing the edges of $G$ until all vertices are visited. At each step, if an unvisited neighbor is encountered, the algorithm adds the corresponding edge to the spanning tree. The process continues until all vertices of $G$ are visited, resulting in a uniformly random spanning tree. 

In our context, $G$ is the complete graph $K_n$ with $n$ vertices (loops included). If we naively run the Aldous-Broder algorithm on $K_n$, we obtain a random Cayley tree with $n$ vertices, but the running time is in $O(n \log(n))$ in the average case:
this is equivalent to the coupon collector's problem.

It is however possible to fasten the random walk by skipping the part where the random walk goes through vertices which are already visited. 
Indeed, whenever the walk goes from a vertex $\vertex u$ to an already visited vertex $\vertex v$, then we can force the random walk to go from $\vertex v$ to an unvisited vertex, chosen at random among the unvisited vertices, without any bias on the spanning tree. This is justified by the following lemma.

\begin{lemma}Consider a portion of a random walk on $K_n$ in the form $\vertex u \rightarrow \vertex{v_1} \rightarrow \dots \rightarrow  \vertex{v_\ell}  \rightarrow \vertex w$, where $\vertex u$ and $\vertex w$ are vertices visited for the first time, and $\vertex{v_1}, \dots, \vertex{v_\ell}$ are already visited vertices.

Then, the random variable associated with $\vertex{v_1}$ has the same probability distribution as the random variable associated with $\vertex{v_\ell}$.
\end{lemma} 
 \begin{proof}The random walk has the same probability to do $\vertex u \rightarrow \vertex{v_1} \rightarrow \dots \rightarrow  \vertex{v_\ell}  \rightarrow \vertex w$ than $\vertex u \rightarrow \vertex{v_\ell} \rightarrow \dots \rightarrow  \vertex{v_1}  \rightarrow \vertex w$ by symmetry of the complete graph.
 This explains why $\vertex{v_1}$ and $\vertex{v_\ell}$ share the same distribution.
 \end{proof}

\newcommand{\algocomment}[1]{\hfill \textcolor{Sepia}{\textbackslash\hspace{-5pt}\textbackslash\ #1}}
\begin{algorithm}
\textbf{Input:} \texttt{n}, a positive integer. \\
\textbf{Output:} \texttt{Edges}, an array listing the edges of a uniform random Cayley tree of size \texttt{n} whose vertex set is $\{1,\dots,\texttt{n}\}$. \\
\hspace*{1cm} $\texttt{Edges} \gets$ empty array; \\
\hspace*{1cm} $\texttt{Vertices} \gets$ $[1,2,\dots,\texttt{n}]$ array;  \algocomment{Arrays are indexed starting from 1} \\
\hspace*{1cm} \textbf{for} \texttt{nb\_unvisited} \textbf{from} $\texttt{n} - 1$ \textbf{to} $1$ \\
\hspace*{0.1cm} \!\algocomment{the subarray \texttt{Vertices[nb\_unvisited$+1\dots$\texttt{n}]} lists the encountered vertices} \\
\hspace*{2cm} \texttt{prev} $\gets$ \texttt{nb\_unvisited+1} \algocomment{index of the current position} \\
\hspace*{2cm} \texttt{next} $\gets$ random number between $1$ and \texttt{n}; \\
\hspace*{2cm} \textbf{if} \texttt{next} $>$ \texttt{nb\_unvisited} \textbf{then} \algocomment{if the next vertex was already visited}\\
\hspace*{3cm} \texttt{prev} $\gets$ \texttt{next}; \algocomment{then the random walk moves forward and} \\
\hspace*{3cm} \ \algocomment{at the next step we force it to go to an unvisited vertex :} \\
\hspace*{3cm} \texttt{next} $\gets$ random number between $1$ and \texttt{nb\_unvisited}; \\
\hspace*{2cm} \textbf{end if} \\
\hspace*{2cm} add $(\texttt{Vertices[prev]},\texttt{Vertices[next]})$ to \texttt{Edges}; \\
\hspace*{2cm} swap \texttt{Vertices[next]} and \texttt{Vertices[nb\_unvisited]};  \\
\hspace*{1cm} \textbf{end for} \\
\hspace*{1cm} \textbf{return} \texttt{Edges}; \\
\caption{Uniform random sampler for Cayley trees of size $n$. \label{algo:cayley}}
\end{algorithm}

Using this observation, Algorithm~\ref{algo:cayley} samples Cayley trees of given size according to the uniform distribution.
It runs in $O(n)$, where $n$ is the number of vertices, even in the worst-case scenario. 

Note that this algorithm performs on average $\frac 3 2 \, n$ random integer draws, whereas an algorithm based on Prüfer codes does only $n$ samples (which is optimal).

\subsection{Random increasing 1,2-trees}

At this point, drawing an increasing 1,2-tree with $n$ vertices uniformly at random is straightforward:
\begin{itemize}
\item we start from a random Cayley tree with $n$ vertices output by Algorithm~\ref{algo:cayley},
\item we apply the bijection $\tau$ of Section~\ref{s:bij} -- or more precisely its reciprocal $\tau^{-1}$ in Subsection~\ref{ss:from_cayley_to_increasing}.
\end{itemize} 
However, we need to be careful with the implementation of certain steps to ensure that we remain effective.
Reversed Step 3 can be performed in linear time with a recursive traversal of the Cayley Tree, remembering the mimima of the subtrees.
Sorting the children with respect to their minima can be done in $O(n)$ too, with a Bucket Sort for example.
Reversed Step 2 is just a rightmost DFS, which also has a linear time complexity.
Reversed Step 1 is more delicate, especially the part where we recover the vertices behind the question marks (see Figure~\ref{recovering_edge}). 
This problem is reminiscent of the All Left Nearest Smaller Neighbour problem~\cite{all_nearest_smaller_values} and can be solved in linear time by gathering all edges which share the same $\textsc{Eddy}$.

\begin{figure}[h!]
  \begin{center}
    \includegraphics[scale=0.5]{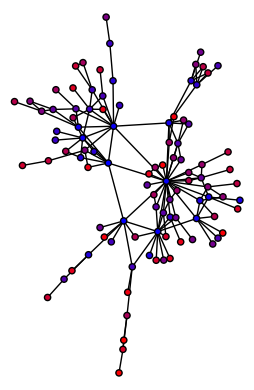}
    \hfill
    \includegraphics[scale=0.5]{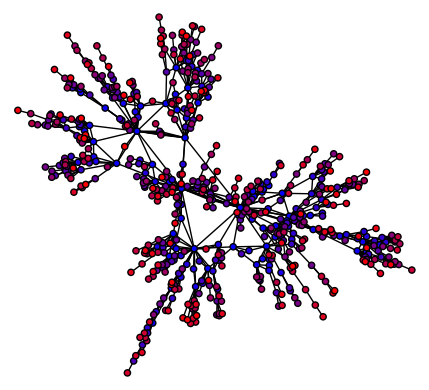}
    \hfill
    \includegraphics[scale=0.5]{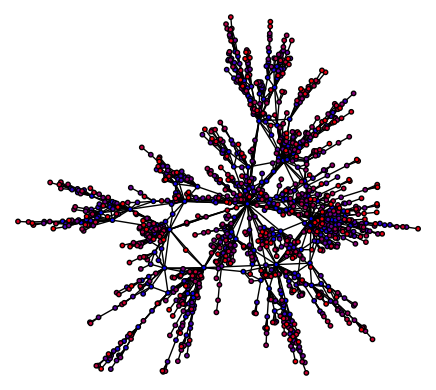}
  \end{center}

  \caption{Random increasing 1,2-trees with sizes 100, 600, 1200, respectively. The redder the vertices, the larger the labels.}
  \label{fig:random}
\end{figure}

Figure~\ref{fig:random} shows some random increasing 1,2-trees drawn uniformly at random.

\section{Conclusion}

By starting our study of increasing \lektrees\  with $k=2$, we discovered that increasing 1,2-trees are equinumerous with Cayley trees. 
We have given three proofs of this result, each with its own advantages:
\begin{enumerate}
  
  \item \textbf{Proof by recursive decomposition.} 
  This proof, mainly based on the Cayley tree decomposition of Shor~\cite{sho95}, is the simplest of the three.
  
  \item  \textbf{Proof by generating function.} 
  We found an exact formula for the genera\-ting function of increasing 1,2-trees (see Theorem~\ref{theo:gf}).
   The method, interesting in itself, involves solving a Partial Differential Equation using the method of characteristics, and should provide a promising approach for the study of \lektrees, with $k \geq 3$. 
   Besides, given how simple the generating function is, it would be intriguing to investigate a combinatorial interpretation (see Remark~\ref{rem:interpretation}).
   Finally, it might be worthwhile to study the asymptotic distribution of the number of triangles in a random increasing 1,2-tree.
   
   \item \textbf{Proof by bijection.} 
   This explanation is the most satisfactory from a combinatorial point of view. 
   It shows that  the increasing edges of Cayley trees are preserved via the bijection, while the twists are somehow transformed into triangles.
   Moreover, we have used this bijection to design an efficient random generator for increasing 1,2-trees.
\end{enumerate}

The investigation of increasing \lektrees\  with $k \geq 3$ should be a natural follow-up of this work.
In particular, it should be interesting to compare uniform random increasing \lektrees\ with other distributions on  \lektrees: 
uniform random labeled \lektrees~ \cite{enumeration-chordal},
\lektrees\ seen as a Pólya urn model\footnote{
  Cliques of size $x$ are represented by balls labeled by $x$. Start with an urn with one ball labeled by $1$. 
  At each step, draw uniformly a ball and put it back in the urn. If it was labeled $x$, add to the urn ${ x \choose i - 1}$ balls labeled $i$ for each $i \in \{1,\dots,x+1\}$. 
  (For example, if it was labeled $1$, you need to add $1$ ball labeled $1$ and $1$ ball labeled 2.
   If it was labeled $2$, you add $1$ ball labeled $1$, $2$ balls labeled $2$, and $1$ ball labeled $3$.)}.
Notably, how does the distribution impact the asymptotic shape of the \lektree? How do the numbers of cliques of given size change? 

\bibliography{biblio.bib}

\begin{thebibliography}{10}

\bibitem{aldous}
D.~J. Aldous.
\newblock The random walk construction of uniform spanning trees and uniform
  labelled trees.
\newblock {\em SIAM J. Discrete Math.}, 3(4):450--465, 1990.

\bibitem{partial-k-tree}
S.~Arnborg, D.~G. Corneil, and A.~Proskurowski.
\newblock Complexity of finding embeddings in a {$k$}-tree.
\newblock {\em SIAM J. Algebraic Discrete Methods}, 8(2):277--284, 1987.

\bibitem{number-k-trees}
L.~W. Beineke and R.~E. Pippert.
\newblock The number of labeled {$k$}-dimensional trees.
\newblock {\em J. Combinatorial Theory}, 6:200--205, 1969.

\bibitem{split}
E.~A. Bender, L.~B. Richmond, and N.~C. Wormald.
\newblock Almost all chordal graphs split.
\newblock {\em J. Austral. Math. Soc. Ser. A}, 38(2):214--221, 1985.

\bibitem{all_nearest_smaller_values}
O.~Berkman, B.~Schieber, and U.~Vishkin.
\newblock Optimal doubly logarithmic parallel algorithms based on finding all
  nearest smaller values.
\newblock {\em J. Algorithms}, 14(3):344--370, 1993.

\bibitem{le_premier_papier_de_matthieu}
O.~Bodini, M.~Dien, A.~Genitrini, and F.~Peschanski.
\newblock The ordered and colored products in analytic combinatorics:
  application to the quantitative study of synchronizations in concurrent
  processes.
\newblock In {\em 2017 {P}roceedings of the {F}ourteenth {W}orkshop on
  {A}nalytic {A}lgorithmics and {C}ombinatorics ({ANALCO})}, pages 16--30.
  SIAM, Philadelphia, PA, 2017.

\bibitem{enumeration-chordal}
J.~Castellv\'{\i}, M.~Drmota, M.~Noy, and C.~Requil{\'e}.
\newblock Chordal graphs with bounded tree-width.
\newblock {\em Adv. in Appl. Math.}, 157:Paper No. 102700, 27, 2024.

\bibitem{limit}
J.~Castellv\'{\i} and B.~Stufler.
\newblock Limits of chordal graphs with bounded tree-width.
\newblock {\em Random Structures Algorithms}, 66(1):Paper No. e21275, 15, 2025.

\bibitem{cayley}
A.~Cayley.
\newblock {\em A theorem on trees}, volume~23.
\newblock Cambridge University Press, 1889.

\bibitem{lambert-function}
R.~M. Corless, G.~H. Gonnet, D.~E.~G. Hare, D.~J. Jeffrey, and D.~E. Knuth.
\newblock On the {L}ambert {$W$} function.
\newblock {\em Adv. Comput. Math.}, 5(4):329--359, 1996.

\bibitem{dieudonne}
J.~Dieudonn{\'e}.
\newblock {\em {\'E}l{\'e}ments d'analyse. {T}ome {IV}: {C}hapitres {XVIII}
  {\`a}{} {XX}}, volume Fasc. XXXIV of {\em Cahiers Scientifiques [Scientific
  Reports]}.
\newblock Gauthier-Villars {\'E}diteur, 1971.

\bibitem{malgaches}
D.~Dumont and A.~Ramamonjisoa.
\newblock Grammaire de {R}amanujan et arbres de {C}ayley.
\newblock volume~3, pages Research Paper 17, approx. 18. 1996.
\newblock The Foata Festschrift.

\bibitem{turcs2}
T.~Ekim, M.~Shalom, and O.~{\c Seker}.
\newblock The complexity of subtree intersection representation of chordal
  graphs and linear time chordal graph generation.
\newblock {\em J. Comb. Optim.}, 41(3):710--735, 2021.

\bibitem{phylogeny}
J.~Enright and G.~Kondrak.
\newblock The application of chordal graphs to inferring phylogenetic trees of
  languages.
\newblock In H.~Wang and D.~Yarowsky, editors, {\em Proceedings of 5th
  International Joint Conference on Natural Language Processing}, pages
  545--552, Chiang Mai, Thailand, Nov. 2011. Asian Federation of Natural
  Language Processing.

\bibitem{methode-caracteristiques}
L.~C. Evans.
\newblock {\em Partial differential equations}, volume~19 of {\em Graduate
  Studies in Mathematics}.
\newblock American Mathematical Society, Providence, RI, second edition, 2010.

\bibitem{flajolet-sedgewick}
P.~Flajolet and R.~Sedgewick.
\newblock {\em Analytic Combinatorics}.
\newblock Cambridge University Press, 2009.

\bibitem{origin-k-trees}
F.~Harary and E.~M. Palmer.
\newblock On acyclic simplicial complexes.
\newblock {\em Mathematika}, 15:115--122, 1968.

\bibitem{polynomial-sampling}
{\'U}.~H{\'e}bert-Johnson, D.~Lokshtanov, and E.~Vigoda.
\newblock Counting and sampling labeled chordal graphs in polynomial time.
\newblock In {\em 31st annual {E}uropean {S}ymposium on {A}lgorithms}, volume
  274 of {\em LIPIcs. Leibniz Int. Proc. Inform.}, pages Art. No. 58, 17.
  Schloss Dagstuhl. Leibniz-Zent. Inform., Wadern, 2023.

\bibitem{bayesian}
I.~{Hob{\ae}k{} Haff}, K.~Aas, A.~Frigessi, and V.~Lacal.
\newblock Structure learning in {B}ayesian networks using regular vines.
\newblock {\em Comput. Statist. Data Anal.}, 101:186--208, 2016.

\bibitem{algo_bidon}
L.~Markenzon, O.~Vernet, and L.~Araujo.
\newblock Two methods for the generation of chordal graphs.
\newblock {\em Annals of Operations Research}, 157(1):47--60, January 2008.

\bibitem{bayesian2}
O.~J. Mengshoel.
\newblock Understanding the scalability of {B}ayesian network inference using
  clique tree growth curves.
\newblock {\em Artificial Intelligence}, 174(12-13):984--1006, 2010.

\bibitem{trois-algos}
A.~Mukhopadhyay and M.~Z. Rahman.
\newblock Algorithms for generating strongly chordal graphs.
\newblock In {\em Transactions on computational science {XXXVIII}}, volume
  12620 of {\em Lecture Notes in Comput. Sci.}, pages 54--75. Springer, Berlin,
  2021.

\bibitem{nijenhuis-wilf}
A.~Nijenhuis and H.~S. Wilf.
\newblock {\em Combinatorial algorithms}.
\newblock Computer Science and Applied Mathematics. Academic Press, Inc.
  [Harcourt Brace Jovanovich, Publishers], New York-London, second edition,
  1978.
\newblock For computers and calculators.

\bibitem{prufer}
H.~Pr{\"u}fer.
\newblock Neuer {Beweis} eines {Satzes} {\"u}ber {Permutationen}.
\newblock {\em Archiv der Mathematischen Physik}, 27:742--744, 1918.

\bibitem{ramanujan}
S.~Ramanujan.
\newblock {\em Notebooks. {V}ols. 1, 2}.
\newblock Tata Institute of Fundamental Research, Bombay, 1957.

\bibitem{linear_system}
D.~J. Rose.
\newblock A graph-theoretic study of the numerical solution of sparse positive
  definite systems of linear equations.
\newblock In {\em Graph theory and computing}, pages 183--217. Academic Press,
  New York-London, 1972.

\bibitem{turcs}
O.~{\c Seker}, P.~Heggernes, T.~Ekim, and Z.~C. {Ta\c sk{\i}n}.
\newblock Generation of random chordal graphs using subtrees of a tree.
\newblock {\em RAIRO Oper. Res.}, 56(2):565--582, 2022.

\bibitem{sho95}
P.~Shor.
\newblock A new proof of {C}ayley's formula for counting labeled trees.
\newblock {\em J. Combin. Theory Ser. A}, 71(1):154--158, 1995.

\bibitem{shor-meir}
P.~Shor, O.~G. Ruehr, and A.~Meir.
\newblock A combinatorial identity ({P}eter {S}hor).
\newblock {\em SIAM Review}, 21(2):258--260, 1979.

\bibitem{markov-chain}
W.~Sun and I.~Bez{\'a}kov{\'a}.
\newblock Sampling random chordal graphs by {M}{C}{M}{C} (student abstract).
\newblock {\em Proceedings of the AAAI Conference on Artificial Intelligence},
  34:13929--13930, 04 2020.

\bibitem{optimal_prufer}
X.~Wang, L.~Wang, Y.~Wu, et~al.
\newblock An optimal algorithm for {P}r{\"u}fer codes.
\newblock {\em J. Softw. Eng. Appl.}, 2(2):111--115, 2009.

\bibitem{wormald}
N.~C. Wormald.
\newblock Counting labelled chordal graphs.
\newblock {\em Graphs Combin.}, 1(2):193--200, 1985.

\end{thebibliography}

\end{document}